\documentclass[11pt]{article}

\usepackage[utf8]{inputenc}
\usepackage[T1]{fontenc}
\usepackage{amsmath,amssymb,amsthm}
\usepackage{graphicx}

\theoremstyle{definition}
\newtheorem{definition}{Definition}

\theoremstyle{plain}
\newtheorem{theorem}{Theorem}
\newtheorem{remark}{Remark}
\newtheorem{lemma}{Lemma}
\newtheorem{proposition}{Proposition}
\newtheorem{corollary}{Corollary}

\newtheorem{construction}{Construction}

\title{Smooth Fractal Trees: Analytic Generators and Discrete Equivalence}
\author{
Henk Mulder\thanks{Independent researcher. Email: \texttt{mulder@skystrategy.com}}
}

\date{}

\begin{document}
\maketitle

\begin{abstract}
\medskip

We introduce a framework for constructing fractal trees via analytic generator
fields, replacing discrete affine transformations and symbolic rewriting rules
by the integration of smooth vector fields defined by ordinary differential
equations in an internal state space. In this setting, geometric curves are
obtained as projections of generator trajectories, and branching is implemented
as a primitive operation through exact inheritance of generator state.

At every finite depth, the resulting structure is a finite union of analytic
curve segments that is smooth across branch events. Two structural results
relate this generator-driven construction to classical discrete models of
tree-based fractals. First, a combinatorial universality theorem shows that any
discrete tree specification, including those arising from iterated function
systems and L-systems, can be compiled into an analytic generator tree whose
induced discrete scaffold is isomorphic at every finite depth. Second, under
standard contractive assumptions, a canopy set equivalence theorem establishes
that the accumulation set of analytic branch endpoints coincides with the
attractor of the corresponding discrete construction.

These results separate local geometric regularity from global fractal
complexity, showing that fractality is determined by recursive branching and
scaling rather than by local non-smoothness. By embedding discrete self-similar
trees into a smooth dynamical systems setting, the framework provides a
structural bridge between combinatorial fractal models and analytic generator
dynamics, while preserving both finite combinatorial structure and asymptotic
limit geometry.
\end{abstract}

\section{Introduction}

Classical fractal geometry is typically developed through discrete generative 
mechanisms: iterated function systems, L-systems, and related symbolic grammars
\cite{mandelbrot1982fractal,hutchinson1981,prusinkiewicz1990}. These frameworks 
generate a nested family of finite trees or polygonal approximants and are most 
often analyzed through the properties of their limiting attractors. In many 
applications, however, the intermediate tree geometry---the finite-depth 
branching structure and the shape of its branches---is as important as the 
limit set.

In an earlier paper \cite{mulder2015} the author observed that discrete tree 
structures can be emulated by trees constructed from analytic curves and 
branchpoints and still produce equivalent fractal limit sets. The present paper 
formalizes this as a generator-first alternative in which tree geometry is 
produced by integrating smooth dynamical processes. We work in a generator 
domain $J\subset\mathbb{R}$ and an internal state space $\mathbb{R}^n$ governed by 
an analytic generator field. A realized branch in $\mathbb{R}^d$ is obtained only 
after projection of the integrated state trajectory. Branching is not an emergent 
geometric event (e.g.\ a kink, attachment, or splice); it is declared as a 
primitive operation implemented by exact state inheritance at prescribed 
generator phases. Consequently, at every finite depth the constructed tree is a 
finite union of analytic curve segments, and no branchpoint singularities are 
introduced by construction.

This paper does not introduce new fractal invariants or dimension-theoretic
results. Its contribution is structural: it provides an analytic realization
framework that preserves the combinatorial growth and asymptotic limit geometry
of classical discrete tree-based fractals.

Our objective is to clarify what is structurally essential in tree-based 
fractality by separating \emph{combinatorial growth} from \emph{local geometric 
realization}. The key question is whether a smooth, generator-driven process can 
match discrete formalisms at the level that determines the limit set.

We prove two main results. First, a \emph{combinatorial universality theorem}: 
any discrete tree-based fractal specification determines a rooted tree with 
per-branch parameters (contraction ratios, orientation labels), and there exists 
an analytic generator tree whose induced discrete scaffold is isomorphic to this 
discrete tree at every finite depth. When an embedding is given, branchpoints can 
be chosen to coincide with the discrete node locations. Second, a \emph{canopy set 
equivalence theorem}: under standard contractive hypotheses, the canopy set 
(the accumulation set of branch endpoints) coincides with the attractor of the 
corresponding discrete construction.

Together, these results show that for tree-based fractals, the global limit set 
is determined by recursive branching and contraction rather than by local 
non-differentiability. Analytic generators therefore constitute a smooth 
compilation target for classical discrete specifications, placing fractal trees
within the scope of differential and dynamical systems tools while preserving 
their finite combinatorial structure and asymptotic limit geometry.

\section{Analytic generators: definitions and framework}

\subsection{Generator space and state space}

Let $J \subset \mathbb{R}$ be a connected open interval, referred to as the \emph{generator domain}.  
Elements $s \in J$ index generative progression and impose a total ordering on generator evolution.  
No geometric meaning (such as arc length or time) is assumed for $s$.

When constructing specific branches, we work with compact subintervals $[s_0,s_1]\subset J$. The open 
domain assumption ensures analyticity extends to branch endpoints, which serve as the locations 
of branch events.

Let $X : J \to \mathbb{R}^n$ denote the \emph{generator state}, with
\[
X(s) = (x_1(s), \dots, x_n(s)).
\]
The state space $\mathbb{R}^n$ is an internal space whose components encode all information required for geometric realization. The choice of $n$ is application-dependent.

\subsection{Generator fields}

\begin{definition}[Analytic Generator Field]
An \emph{analytic generator field} is a map
\[
V : J \times \mathbb{R}^n \to \mathbb{R}^n
\]
that is analytic (or $C^k$, $k \ge 1$) in both arguments.
\end{definition}

A generator field defines the ordinary differential equation
\begin{equation}
\label{eq:generator_ode}
\frac{dX}{ds} = V(s, X),
\end{equation}
together with an initial condition
\[
X(s_0) = X_0 \in \mathbb{R}^n.
\]

\begin{proposition}[Local Realizability]
\label{prop:local_realizability}
Let $V : J \times \mathbb{R}^n \to \mathbb{R}^n$ be an analytic (or $C^k$, $k \ge 1$) generator field, and let $(s_0, X_0) \in J \times \mathbb{R}^n$ be an initial condition.  
Then there exists $\varepsilon > 0$ and a unique solution
\[
X : (s_0 - \varepsilon,\, s_0 + \varepsilon) \to \mathbb{R}^n
\]
to the ordinary differential equation \eqref{eq:generator_ode} satisfying $X(s_0) = X_0$
(cf.\ Picard--Lindel\"of).
\end{proposition}

This solution is referred to as a \emph{generator realization}. All geometry in this framework arises as a derived object from $X(s)$.

\begin{remark}[Role of Analyticity.]
Generator fields are assumed analytic throughout this work. This assumption has two distinct roles:

First, it ensures \textit{maximal regularity} of realized branches and compatibility with the full toolkit of analytic methods from 
differential geometry and dynamical systems. The analytic structure of generator-driven fractals is a key distinguishing feature: it 
enables, in principle, local power series expansions, complex-analytic extensions, 
and differential operator techniques that are unavailable for discrete or piecewise-linear fractal constructions.

Second, from a purely \textit{structural} perspective, analyticity is not essential: all main results---branch inheritance, finite-stage 
smoothness, combinatorial universality (Theorem~1), and canopy set equivalence (Theorem~2)---remain valid for generator fields of class 
$C^k$, $k \geq 1$, under standard existence and uniqueness assumptions for ODEs.

We work with analytic generators to fully exploit the differential structure this framework affords, while noting that the core 
geometric and combinatorial mechanisms operate under weaker regularity assumptions.
\end{remark}

\subsection{Geometric realization}

Let $\Pi : \mathbb{R}^n \to \mathbb{R}^d$ be a projection or interpretation map selecting geometric components of the state. The realized curve is defined by
\begin{equation}
\label{eq:realized_curve}
\gamma(s) = \Pi(X(s)) \in \mathbb{R}^d.
\end{equation}

The choice of $\Pi$ is fixed for a given construction and is not part of the generator dynamics.  
Geometry is therefore \emph{realized}, not prescribed.

\subsection{Canonical planar generator example}

As a concrete example, consider a planar curve with curvature-driven evolution.  
Let
\[
X(s) = (x(s), y(s), \theta(s)) \in \mathbb{R}^3,
\]
where $(x,y)$ represent position and $\theta$ represents orientation.

Define the generator field
\begin{equation}
\label{eq:planar_generator}
\begin{aligned}
\frac{dx}{ds} &= \rho(s)\cos\theta(s), \\
\frac{dy}{ds} &= \rho(s)\sin\theta(s), \\
\frac{d\theta}{ds} &= \kappa(s),
\end{aligned}
\end{equation}
where $\rho(s)$ and $\kappa(s)$ are analytic scalar functions.

We assume $\rho(s) > 0$ for all $s \in J$, ensuring that the realized curve is regular.

The realized curve is obtained via the projection
\[
\Pi(x,y,\theta) = (x,y).
\]

This generator defines a smooth planar curve whose local geometry is controlled analytically by $\rho$ and $\kappa$. Global geometric complexity arises solely from the structure of these generator functions and from branching operations introduced later.

This representation reflects the classical fact that a regular planar curve
is uniquely determined (up to rigid motion) by its curvature and speed functions
\cite{docarmo2016}.

\subsection{Generator--geometry separation}

Equations \eqref{eq:generator_ode}--\eqref{eq:realized_curve} formalize the separation between:
\begin{itemize}
\item generator evolution in state space,
\item and geometric realization in embedding space.
\end{itemize}

At no point is geometry specified directly. All geometric properties of $\gamma$ are consequences of the analytic generator field and its integration.

This generator-first formulation underlies all subsequent constructions in this work.

\section{Branching as a primitive operation}
Branching is treated as a primitive operation of the framework, rather than as an emergent consequence of the generator dynamics or geometric singularities.

\begin{definition}[Branch Event]
A \emph{branch event} occurs at a generator parameter value $s=s_b \in J$
when the realization $X$ spawns a finite family of child realizations
$\{X_i\}_{i=1}^m$, for some \emph{branch multiplicity} $m \in \mathbb{N}$, each defined on its own
generator interval $J_i \subset \mathbb{R}$ with distinguished origin
$0 \in J_i$, such that
\begin{equation}
\label{eq:state_inheritance}
X_i(0) = X(s_b), \qquad i=1,\dots,m.
\end{equation}
\end{definition}

Equation \eqref{eq:state_inheritance} is the \emph{state inheritance rule}. It asserts that branching is implemented as continuation of the internal state, not as geometric attachment. In particular, there is no additional branch-dependent offset term added to $X_i(0)$.

\begin{definition}[No-Offset Branching Constraint]
A branching rule satisfies the \emph{no-offset constraint} if for every child realization $X_i$ spawned at $s_b$,
\[
X_i(0) = X(s_b)
\]
holds exactly, i.e.\ no branch-dependent translation, rotation, or angular offset is applied at the event.
\end{definition}

When the state $X$ includes orientation variables (e.g.\ $\theta$ in \eqref{eq:planar_generator}), the no-offset constraint implies that the tangent direction at the start of each child branch is inherited:
\[
\theta_i(0) = \theta(s_b).
\]
Thus, branching does not introduce geometric kinks by construction; any divergence in geometry arises only through subsequent generator evolution.

\subsection{Generator inheritance and modification}

Each child realization $X_i$ evolves according to an assigned child generator field
\[
V_i : J_i \times \mathbb{R}^n \to \mathbb{R}^n,
\qquad
\frac{dX_i}{ds} = V_i(s, X_i),
\qquad
X_i(0)=X(s_b).
\]

\begin{definition}[Generator Inheritance]
A child generator field $V_i$ \emph{inherits} the parent generator field $V$ if there exists a reparameterization $\phi_i : J_i \to J$ with $\phi_i(0)=s_b$ such that
\begin{equation}
\label{eq:generator_inheritance}
V_i(s, x) = \dot{\phi}_i(s)\, V(\phi_i(s), x)
\end{equation}
for all $(s,x)$ in the domain of $V_i$.
\end{definition}

In many constructions, children inherit the same functional form of generator data but with controlled modifications (e.g.\ sign changes, scaling, or parameter shifts). The following planar specialization makes this explicit.

\paragraph{Planar specialization.}
For the canonical planar state $X=(x,y,\theta)$ and generator \eqref{eq:planar_generator}, define child generators by
\begin{equation}
\label{eq:binary_inheritance}
\rho_i(s) = \lambda_i\,\rho(s), 
\qquad
\kappa_i(s) = \sigma_i\,\kappa(s),
\end{equation}
where $\lambda_i>0$ is a scaling factor and $\sigma_i \in \{+1,-1\}$ selects left/right turning. Each child then evolves by
\[
\frac{dx_i}{ds}=\rho_i(s)\cos\theta_i,\quad
\frac{dy_i}{ds}=\rho_i(s)\sin\theta_i,\quad
\frac{d\theta_i}{ds}=\kappa_i(s),
\quad
X_i(0)=X(s_b).
\]
Binary branching corresponds to $m=2$ with $\sigma_1=+1$ and $\sigma_2=-1$; multi-branching corresponds to $m>2$ with distinct parameter choices $(\lambda_i,\sigma_i)$ or more general modifications of $(\rho,\kappa)$.

\begin{proposition}[Regularity Across Branch Events]
Assume each $V_i$ is analytic (or $C^k$) and the no-offset constraint holds. Then each child realization $X_i$ is analytic (resp.\ $C^k$) on $J_i$, and the realized geometry $\gamma_i(s)=\Pi(X_i(s))$ is continuous at the branch event in the sense that
\[
\gamma_i(0)=\gamma(s_b)
\]
for all children $i$.
\end{proposition}

No additional continuity conditions are imposed; the branch event is defined entirely in generator space by \eqref{eq:state_inheritance}.

\begin{remark}[Local reset versus global generator progress]
\label{rem:local-vs-global-progress}
At a branch event, each child realization is naturally parameterized by a local
generator coordinate $s$ with origin at the branch point. This local reset is
sufficient for generators whose behavior depends only on local structure.

However, certain generator terms---such as global shrinkage profiles, environment
fields, or accumulated effects---require continuity of an absolute notion of
generator progress across branch events. This can be achieved either by 
reparameterizing the child generator via a map $\phi_i(s)$ satisfying 
$\phi_i(0)=s_b$, or by introducing a global progress variable $\tau$ inherited 
unchanged at branching.

In the latter formulation, all branches share the same generator phase $\tau$, 
and shrinkage or decay is encoded directly in the shared generator profile 
$\rho(\tau)$. Branch-specific scaling factors are then unnecessary ($\lambda_i = 1$), 
with branching effects arising solely through state inheritance and discrete 
generator modifications (e.g.\ sign changes in curvature).
\end{remark}

The inheritance rule \eqref{eq:generator_inheritance} is consistent with standard 
reparameterizations of vector fields in ordinary differential equation theory, 
which preserve solution orbits under changes of the evolution parameter 
(see \cite{hirsch_smales_devaney}). It also parallels the reuse of a common 
functional template with parameterized variants in iterated function systems 
\cite{hutchinson1981,falconer2003fractal}, with the key distinction that inheritance 
here acts on generator fields rather than geometric maps and introduces no 
geometric offsets.
\subsection{Tree structure in generator space}

A branch event at $s_b$ induces a rooted tree structure on the set of realizations. Formally, let $\mathcal{N}$ denote the set of all realized branches generated by repeated application of branch events starting from a distinguished root realization $X^{(0)}$. Define a directed edge
\[
X^{(p)} \to X^{(c)}
\]
whenever $X^{(c)}$ is spawned as a child of $X^{(p)}$ at some branch parameter value $s_b$.

\begin{definition}[Generator Tree]
The directed graph $(\mathcal{N},\mathcal{E})$, where $\mathcal{E}\subset \mathcal{N}\times\mathcal{N}$ is the parent--child relation induced by branch events, is called the \emph{generator tree}. Its topology is determined entirely by the event schedule $\{s_b\}$ and the branching multiplicities $m$, independent of the geometric embedding of realized branches.
\end{definition}

The realized geometric tree is the embedded image
\[
\mathcal{T} \;=\; \bigcup_{X \in \mathcal{N}} \; \Pi\!\left(X(J_X)\right) \;\subset\; \mathbb{R}^d,
\]
where $J_X$ denotes the generator interval associated with branch $X$. Thus, branching is organized in generator space, and geometric complexity arises as a consequence of (i) the generator dynamics on each branch and (ii) the combinatorial structure of the generator tree, rather than from geometric attachment rules or symbolic grammars.

In particular, the role played by rewriting rules in classical tree grammars is assumed here by explicit branch events together with generator inheritance and modification. The generative mechanism remains analytic throughout.

\section{Construction of analytic fractal trees}

\subsection{Elementary generator examples}

We begin by specifying elementary analytic generator fields that serve as building blocks for fractal tree constructions. All examples are formulated within the planar state space
\[
X(s) = (x(s), y(s), \theta(s)) \in \mathbb{R}^3,
\]
with geometric realization $\gamma(s) = (x(s), y(s))$.

\paragraph{Constant curvature generators.}
Let $\rho_0 > 0$ and $\kappa_0 \in \mathbb{R}$ be constants. Define
\begin{equation}
\label{eq:constant_curvature}
\begin{aligned}
\frac{dx}{ds} &= \rho_0 \cos \theta, \\
\frac{dy}{ds} &= \rho_0 \sin \theta, \\
\frac{d\theta}{ds} &= \kappa_0 .
\end{aligned}
\end{equation}
The realized curve is a circular arc (or straight line if $\kappa_0=0$). 
This generator provides a locally smooth, curvature-controlled primitive 
from which more complex structures are assembled. An example is shown in Figure 1.

\begin{figure}[htbp]
  \centering
  \includegraphics[width=\textwidth]{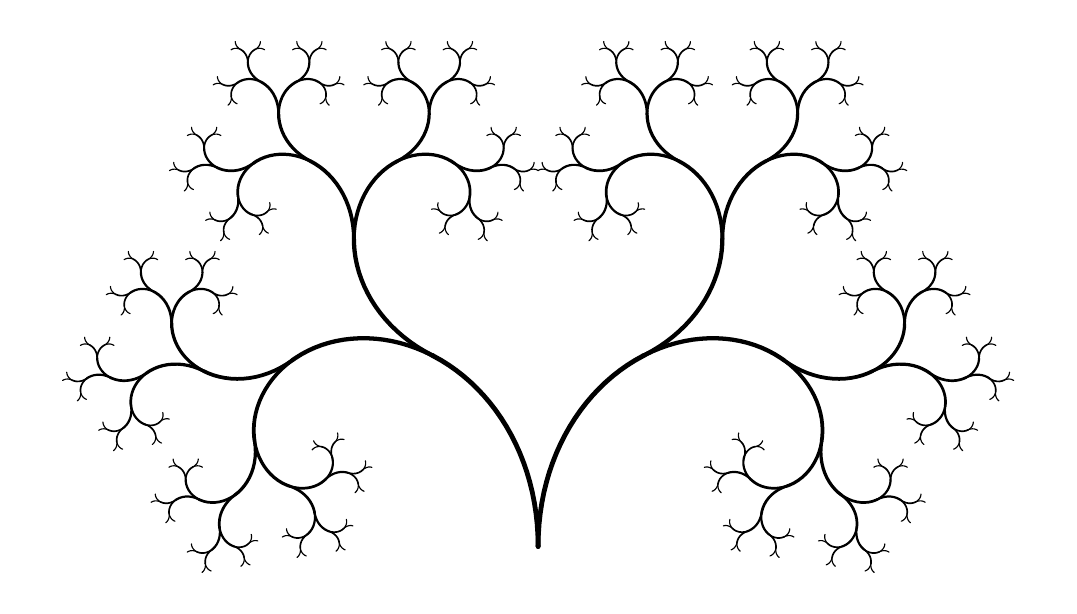}
  \caption{Analytic fractal tree constructed via generator-driven dynamics with exponential 
  radial decay $\rho(s) = A^s$ ($A = 0.88$) and linear angular modulation 
  $\theta(s) = \theta_0 \pm \Omega s$ ($\Omega = \pi/10$). Binary branching is implemented 
  through exact state inheritance at fixed generator phase $s_b$, with left and right children 
  assigned opposite curvature signs ($\sigma_1 = +1$, $\sigma_2 = -1$). All branches are analytic 
  curves; fractality emerges from recursive branching structure rather than local geometric 
  irregularity.}
  \label{fig:smooth_fractal_tree}
\end{figure}

\paragraph{Analytic branch modulation.}
Let $\rho(s)$ and $\kappa(s)$ be analytic functions, for example
\[
\rho(s) = \rho_0, \qquad \kappa(s) = \kappa_0 + \varepsilon \sin(\omega s),
\]
with $\varepsilon,\omega \in \mathbb{R}$. The resulting generator field produces 
smooth oscillatory curvature while preserving analyticity. Such oscillatory 
generators demonstrate that local geometric complexity can be introduced 
continuously within a branch, without recourse to discrete rewriting 
or loss of smoothness, and independently of the branching structure.

\paragraph{Controlled scaling and rotation.}
Global tapering or expansion is introduced by allowing $\rho$ to vary analytically, e.g.
\[
\rho(s) = \rho_0 \alpha^s, \qquad \alpha>0,
\]
with the understanding that continuity across branch events is enforced via 
reparameterization or state augmentation as discussed in 
Remark~\ref{rem:local-vs-global-progress}. Scaling laws of this form control 
branch length hierarchically. 

\subsection{Recursive branching with analytic generators}

The preceding sections define branching and generator inheritance at the level
of individual branch events. We now describe recursive application of these rulesto construct analytic fractal trees. This recursive
viewpoint is essential: fractality arises only through repeated application of
smooth generator evolution and branch inheritance, while local regularity is
preserved at every finite stage of the construction.

Let $X^{(0)}$ denote a root realization defined on a generator interval
$J^{(0)}=[0,S]$ with generator field $V^{(0)}$. A fractal tree is constructed by
recursively applying branch events at a prescribed sequence of generator values
\[
\{s_b^{(k)}\}_{k \ge 1},
\]
with each branch event spawning $m_k \ge 2$ child realizations.

At a branch event $s=s_b$, each child realization $X_i$ inherits the parent state
according to
\[
X_i(0) = X(s_b),
\]
and evolves under a child generator field $V_i$ obtained from the parent
generator by analytic modification. For planar generators of the form
\eqref{eq:constant_curvature}, a typical choice is
\begin{equation}
\label{eq:recursive_branching}
\rho_i(s) = \lambda_i \rho(s), \qquad
\kappa_i(s) = \sigma_i \kappa(s),
\end{equation}
where $\lambda_i \in (0,1)$ controls scale reduction and
$\sigma_i \in \{+1,-1\}$ selects turning direction.

The same branching rule is then applied to each child realization at subsequent
branch parameters, generating a rooted tree of generator trajectories. The
branching schedule and scaling factors determine the combinatorial structure of
the tree, while the generator fields determine the local geometry along each
branch.

\begin{proposition}[Finite-Stage Smoothness]
At any finite branching depth, the realized geometry of an analytic fractal tree
is a finite union of analytic curves. In particular, all branches are everywhere
smooth and free of singularities.
\end{proposition}

\begin{proof}
Each branch segment is the solution of an analytic ordinary differential equation
with analytic initial data. By standard ODE theory, each such solution is analytic
on its domain. Since only finitely many branch events occur at any finite depth,
the realized geometry is a finite union of analytic curves. Continuity at branch
points follows directly from the state inheritance rule.
\end{proof}

\subsection{Local regularity and global fractality}

Each individual branch generated by the constructions above is locally regular: curvature, tangent direction, and higher derivatives remain finite and well-defined at all finite generator values. There is no local metric irregularity or non-differentiability at any point along a branch.

Nevertheless, as the depth of recursive branching increases, the global geometry exhibits increasing structural complexity. The accumulation of infinitely many branch endpoints gives rise to a nontrivial limit set, while local smoothness is preserved at every finite stage. Fractality thus arises from the organization of branching events in generator space rather than from local geometric roughness.

This separation between local regularity and global complexity is a defining feature of analytic fractal trees and underlies the equivalence results established in the following section.

\section{Limit sets and canopy structure}

\subsection{Definition of the canopy set}
Let $\mathcal{T}$ denote an analytic fractal tree constructed by recursive branching as described in Section~4.  
Each branch realization is given by a generator trajectory
\[
X^{(i)} : J^{(i)} \to \mathbb{R}^n,
\]
with realized geometric image
\[
\gamma^{(i)}(s) = \Pi(X^{(i)}(s)) \subset \mathbb{R}^d.
\]
Let $E^{(i)} \subset \mathbb{R}^d$ denote the set of branch endpoints corresponding to the terminal parameter value of $J^{(i)}$.

\begin{definition}[Canopy set]
\label{def:canopy}
The canopy set $\mathcal{C} \subset \mathbb{R}^d$ of an analytic fractal tree 
is defined as
\[
\mathcal{C} = \overline{\bigcup_{i} E^{(i)}},
\]
where the closure is taken in $\mathbb{R}^d$.
\end{definition}

The canopy set captures the asymptotic geometry of the tree independently of any finite truncation depth. When the branching scales satisfy a uniform contraction condition, $\mathcal{C}$ is compact. This notion is directly analogous to the limit set (attractor) associated with discrete fractal constructions.

\subsection{Induced discrete scaffold}

At any finite branching depth, the analytic generator construction induces a discrete rooted tree whose 
nodes are precisely the branchpoints produced by state inheritance. Parent--child relations are determined 
by the branching schedule in generator space.

\begin{figure}[htbp]
  \centering
  \includegraphics[width=\textwidth]{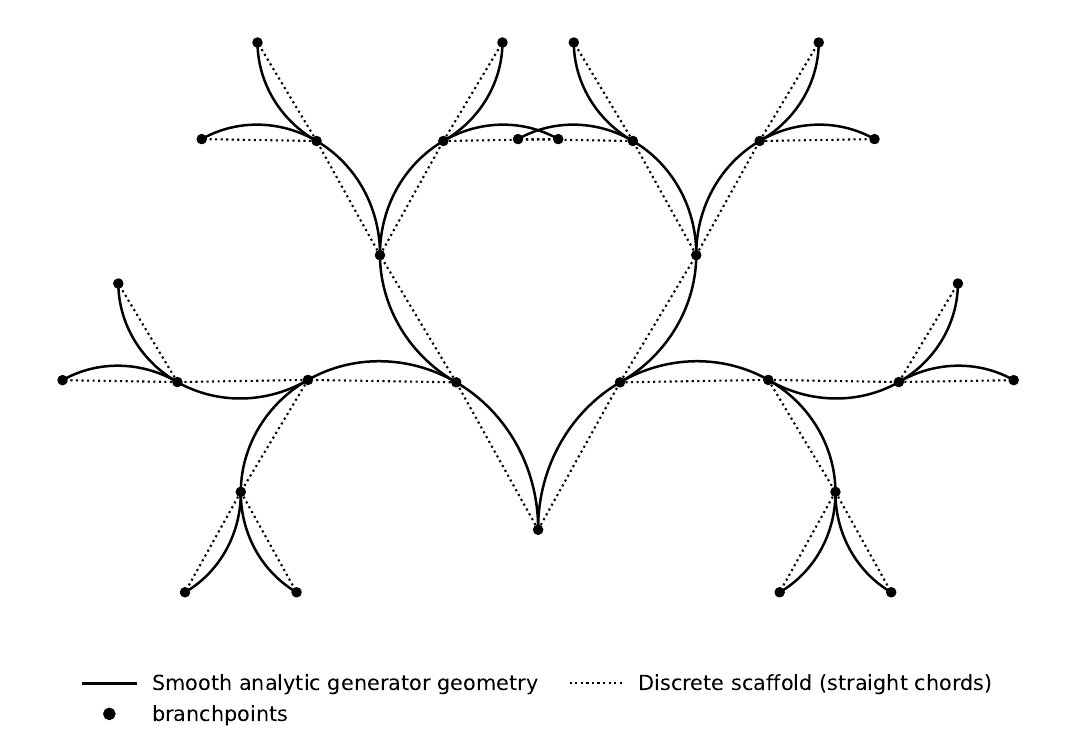}
  \caption{Analytic generator tree (solid curves) and its induced discrete
  scaffold (dotted segments). Branchpoints (filled circles) are realized 
  directly by the generator construction. The discrete scaffold connects 
  these branchpoints with straight chords, preserving combinatorial 
  structure and endpoint locations.}
  \label{fig:scaffold_interpolation}
\end{figure}

\begin{proposition}[Induced discrete scaffold]
At each finite branching depth, the branchpoints of an analytic fractal tree
define a discrete rooted tree whose nodes and parent--child relations coincide
with those of the discrete fractal construction determined by the same branching
multiplicities, contraction ratios, and orientation rules.

In particular, the induced discrete scaffold is combinatorially identical
(i.e. isomorphic as rooted tree) to the corresponding discrete construction,
although the geometric realizations of the edges differ.
\end{proposition}

\begin{proof}
By construction, each branchpoint corresponds to a branch event at which state 
inheritance occurs. The parent-child relations are determined solely by the 
branching schedule, independent of geometric realization. Since the branching 
multiplicities, contraction ratios, and orientation rules match those of the 
discrete construction, the induced graph structure is isomorphic as a rooted tree.
\end{proof}

Each analytic generator segment provides a smooth interpolation between adjacent 
scaffold nodes. Importantly, the generator dynamics do not alter the locations of 
scaffold nodes; they refine only the geometry along the connecting edges. This 
relationship is illustrated in Figure~\ref{fig:scaffold_interpolation}, where the 
discrete scaffold is shown together with its analytic interpolation.

\subsection{Universality and Compilation}\label{sec:universality}

Classical discrete approaches to tree-like fractals —-- most notably iterated 
function systems and grammar-based formalisms such as L-systems —-- generate 
geometry through repeated application of contraction maps or symbolic 
transformations \cite{hutchinson1981,falconer2003fractal,prusinkiewicz1990}. 
In these frameworks, branching structure and geometric realization are specified 
simultaneously: rewriting rules or contraction maps prescribe both the 
combinatorial growth of the tree and the geometry of its branches.

Analytic fractal trees decouple these roles. Branching structure is specified
entirely in generator space through an explicit schedule of branch events and
exact state inheritance, while geometric realization arises only through
integration of smooth generator fields. As a result, each branch segment is 
everywhere analytic at all finite stages of construction, yet the induced 
discrete scaffold reproduces the full combinatorial structure of the 
corresponding discrete construction.

\begin{theorem}[Combinatorial universality of analytic generator trees]
\label{thm:combinatorial_universality}
Let $D$ be any discrete tree-based fractal specification --— such as an
IFS-generated tree or a turtle-interpreted L-system —-- which determines a rooted
tree together with per-branch parameters (e.g.\ contraction ratios and
orientation rules). Then there exists an analytic fractal tree $T$ whose induced
discrete scaffold is combinatorially identical (isomorphic as a rooted tree) to
the discrete construction of $D$ at every finite depth.

Moreover, if the discrete construction of $D$ admits an embedded realization of
its nodes in $\mathbb{R}^d$, the analytic generator tree $T$ may be chosen so that
its branchpoints coincide with these node locations at every finite stage.
\end{theorem}

The proof of Theorem~1 is constructive and follows directly from the
definitions of analytic generator trees and discrete tree specifications.
The key observation is that the branching structure and per-branch
parameters of $D$ determine a rooted tree together with an associated
branching schedule and generator inheritance rules. An analytic generator
tree is then constructed by implementing these branching events via exact
state inheritance and by defining, on each branch, generator fields that
realize the prescribed per-branch parameters. The resulting analytic
fractal tree $T$ induces a discrete scaffold that is combinatorially
identical to the discrete construction of $D$ at every finite depth.
The full proof is given in Appendix~A.

\paragraph{Compilation into analytic generators.}
Theorem~1 is constructive in the sense of existence. Given a discrete tree
specification $D$, one may extract its rooted tree structure and, when
available, its node embedding, and realize each discrete edge by an analytic
curve segment connecting the corresponding parent and child nodes. These
segments define analytic generator data along each branch, while branching is
implemented by exact state inheritance at branchpoints. Iteration over the
discrete tree yields an analytic generator tree whose induced scaffold matches
that of $D$ exactly, while leaving the local branch geometry unconstrained
beyond smoothness.

The correspondence established by Theorem~1 does not privilege a particular
direction of construction. An analytic generator tree may equally well be
taken as primitive, with its induced discrete scaffold recovering the same
rooted tree structure and per-branch similarity data. The following example
illustrates this correspondence in the reverse direction, starting from an
explicit analytic generator construction.

\paragraph{Example 1 (Compilation of a simple binary discrete tree).}
We illustrate the compilation procedure on a minimal discrete tree-based
fractal: a symmetric binary tree with uniform contraction and fixed turning
angles. The purpose of the example is not to introduce new geometry, but to make
explicit how a classical discrete specification is translated into analytic
generator data within the present framework.

\emph{Discrete specification.}
Consider a rooted binary tree in the plane generated by two contractive maps
\[
F_\pm(x) = \lambda R(\pm\theta)\,x + t,
\]
where $0<\lambda<1$, $R(\pm\theta)$ denotes planar rotation by angles $\pm\theta$,
and $t\in\mathbb{R}^2$ is a fixed translation. Iteration of these maps produces a
discrete tree whose nodes are the images of the root under all finite
compositions of $F_+$ and $F_-$.

\begin{figure}[t]
\centering
\includegraphics[width=\linewidth]{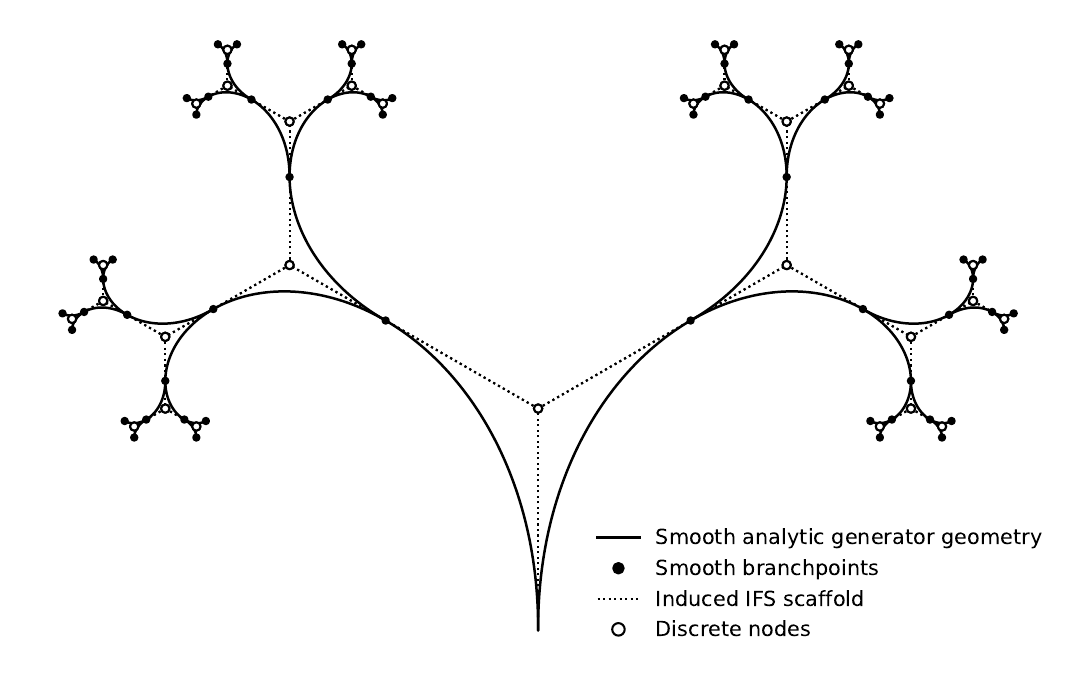}
\caption{Geometric extraction of discrete IFS parameters from a smooth
analytic generator tree. Scaffold nodes (open circles) are obtained by
tangent intersection rather than direct construction. The recovered
contraction ratios $\hat{\lambda}_a = A$ and turning angles $|\hat{\theta}_a| = \omega$ confirm that
discrete similarity data emerges from smooth generator geometry, illustrating 
recovery of discrete similarity data, consistent with Theorem~2.}
\label{fig:smooth-vs-discrete}
\end{figure}

This specification determines (i) a rooted binary tree with branching
multiplicity $m=2$, (ii) a per-branch contraction ratio $\lambda$, and (iii)
per-branch orientation labels $\sigma\in\{+,-\}$. An explicit embedding of the
discrete nodes is obtained by evaluating the compositions of the maps $F_\pm$.

\emph{Analytic compilation.}
We construct an analytic generator tree whose induced discrete scaffold
coincides with this discrete specification.

We work in the canonical planar state space
\[
X(s)=(x(s),y(s),\theta(s))\in\mathbb{R}^3,
\]
with realization map $\Pi(x,y,\theta)=(x,y)$.

Let the root branch be generated on an interval $[0,S]$ by a constant-curvature
generator
\[
\frac{dx}{ds}=\rho_0\cos\theta,\qquad
\frac{dy}{ds}=\rho_0\sin\theta,\qquad
\frac{d\theta}{ds}=0,
\]
so that the realized root branch is a straight segment of length $\rho_0 S$.

At the branch parameter $s=S$, a binary branch event is declared. Two child
realizations $X_+$ and $X_-$ are spawned with exact state inheritance
\[
X_\pm(0)=X(S),
\]
and evolve under modified generator data
\[
\rho_\pm(s)=\lambda\rho_0,\qquad
\kappa_\pm(s)=\pm\kappa_0,
\]
where the constant curvature magnitude $\kappa_0$ is chosen so that the total
turning angle over each child branch equals $\pm\theta$.

This procedure is repeated recursively at the same generator phase on each
branch, yielding a rooted binary generator tree. At every branch event, no
geometric offset is introduced: all child branches inherit position and
orientation exactly, and divergence arises solely through subsequent generator
evolution.

\emph{Correspondence with the discrete tree.}
At each finite depth $k$, the branchpoints of the analytic construction induce a
discrete scaffold obtained by intersecting tangent extensions of parent and
child branches. Parent–child relations of these scaffold nodes are identical to
those of the discrete tree generated by the maps $F_\pm$, and the induced
discrete scaffold is therefore isomorphic to the discrete specification at every
finite stage.

The analytic generator segments provide smooth interpolations between adjacent
scaffold nodes. Their detailed geometry is not fixed by the discrete
specification and may be chosen freely (subject to analyticity), without
affecting the combinatorial structure or the similarity data of the scaffold.

\emph{Worked instantiation and parameter recovery.}
In a concrete instantiation of this construction, branch realizations are
generated by integrating
\[
\frac{dx}{ds}=A^{s}\cos(\theta_0+\omega s),\qquad
\frac{dy}{ds}=A^{s}\sin(\theta_0+\omega s),\qquad
\frac{d\theta}{ds}=\omega,
\]
with fixed segment length $S=1$, contraction parameter $0<A<1$, and angular rate
$\omega>0$. At each branch event, two children are spawned by exact state
inheritance together with the sign choice $\omega\mapsto\pm\omega$.

From the resulting smooth realization, a discrete scaffold is extracted
projectively by intersecting tangent lines at successive branchpoints. For each
generation $g\ge2$, the measured ratios of consecutive scaffold edge lengths and
the turning angles between parent and child edges recover
\[
\widehat{\lambda}_g=A,\qquad |\widehat{\theta}_g|=\omega,
\]
uniformly across all branches, confirming that the similarity data of the
discrete specification emerge directly from the smooth generator geometry.

\emph{Limit geometry.}
Under the contraction hypothesis $0<\lambda<1$, the accumulation set of analytic
branch endpoints (the canopy set) coincides with the attractor of the discrete
system $\{F_\pm\}$. The smooth generator geometry therefore realizes the same
limit set as the discrete construction, while differing locally in its
interpolating curves. This example concretely illustrates the compilation
asserted by Theorem~1 and the canopy set equivalence established in
Theorem~2.

\begin{theorem}[Canopy set equivalence]
\label{thm:canopy_equivalence}
\label{thm:canopy_equivalence}
Under matching branching structure and contraction ratios, the canopy set
$\mathcal{C}$ of the analytic fractal tree $\mathcal{T}$ coincides with the
attractor $\mathcal{A}$ of the corresponding discrete fractal construction:
\[
\mathcal{C} = \mathcal{A}.
\]
\end{theorem}

\begin{proof}[Proof.]
At each finite branching depth, the analytic generator construction induces a
discrete scaffold whose nodes and parent--child relations coincide with those of
the corresponding discrete system. Consequently, the endpoint sets at each depth
are identical in the analytic and discrete constructions. Under uniform
contraction assumptions, these finite-depth endpoint sets converge in the
Hausdorff metric to a unique compact limit, namely the discrete attractor
$\mathcal{A}$; this convergence follows standard arguments from fractal geometry
\cite{falconer2003fractal,hutchinson1981}. The canopy set $\mathcal{C}$, defined as
the accumulation set of analytic branch endpoints, therefore coincides with
$\mathcal{A}$. A complete proof is given in Appendix~B.

\end{proof}

\subsection{Compilation preserves attractors}
\label{sec:compilation-preserves-attractors}

These results show analytic compilation preserves both finite 
branching structure and asymptotic limit geometry.

\begin{corollary}[Compilation preserves attractors]
\label{cor:universality_implies_canopy_equivalence}
Let $D$ be a discrete tree-based fractal specification governed by a finite family
of contractive maps $\{F_i\}_{i=1}^m$ with
$\max_i \mathrm{Lip}(F_i) < 1$, and let $\mathcal{A}$ denote its attractor. Then
there exists an analytic fractal tree $\mathcal{T}$ obtained by compiling $D$
into analytic generators such that the canopy set $\mathcal{C}$ of $\mathcal{T}$
satisfies
\[
\mathcal{C} = \mathcal{A}.
\]
\end{corollary}

\begin{proof}
By Theorem~\ref{thm:combinatorial_universality}, the discrete tree of $D$ can be
compiled into an analytic generator tree $\mathcal{T}$ whose induced scaffold is
combinatorially identical at every finite depth. The hypotheses of
Theorem~\ref{thm:canopy_equivalence} therefore apply, and the canopy set of
$\mathcal{T}$ coincides with the discrete attractor $\mathcal{A}$.
\end{proof}

\section{Implications and Interpretation}

\subsection{Smoothness, fractality, and organizational structure}

Classical fractal constructions are commonly associated with local geometric
irregularity or non-differentiability, features that arise naturally from
discrete rewriting systems or piecewise affine attachment rules. In such
frameworks, fractality is often interpreted as a consequence of local metric
roughness.

The analytic generator constructions developed here show that this association
is not fundamental. At every finite stage, an analytic fractal tree consists of
a finite union of smooth curves generated by analytic ordinary differential
equations. Curvature and tangent direction remain well defined along each
branch, and continuity at branch events is enforced by exact state inheritance
in generator space.

Despite this local regularity, the global limit geometry is fractal in the
classical sense. The accumulation of infinitely many branch endpoints gives rise
to a nontrivial canopy set which, under standard contraction assumptions,
coincides with the attractor of a corresponding discrete construction.
Fractality therefore arises not from local geometric roughness, but from the
recursive organization of branching and scaling.

This phenomenon is consistent with other settings in which fractal limit sets
emerge from smooth dynamics, notably in hyperbolic dynamical systems, where
invariant sets may be geometrically fractal despite smooth governing equations
\cite{katok1995}. Related analytic approaches study differential operators and energy
forms defined directly on fractal limit sets, as in the theory of analysis on
self-similar spaces \cite{kigami2001}. The present framework adopts a complementary
perspective: analytic structure precedes the fractal geometry, which emerges
asymptotically from smooth generator-driven evolution rather than serving as its
domain of definition.

\subsection{Generator-first geometry}

The analytic generator framework adopts a generator-first viewpoint in which
geometry is not specified directly but realized through the integration of
generator fields. Generator dynamics in state space constitute the primary
causal mechanism, while geometric curves arise only as derived objects through
projection.

Within this formulation, branching is implemented through exact inheritance of
generator state rather than through geometric attachment or symbolic rewriting.
As a consequence, continuity and regularity of realized geometry follow directly
from the analytic structure of the generators, without the need for auxiliary
compatibility conditions at branchpoints.

This separates generative causation from geometric realization. Local geometric
properties reflect analytic features of the generator field, while global
structure is determined by the organization of branch events in generator
space. The framework therefore shifts attention from prescribed form to
generative dynamics, with branching acting as a structural operator on generator
trajectories rather than as a geometric splice operation.

\subsection{Contribution and scope}

The contribution of this work is representational and structural rather than
classificatory. No new fractal attractors are introduced, and no new convergence
phenomena are claimed beyond those established in classical discrete theory.
Instead, the results demonstrate that tree-based fractal constructions admit a
fully analytic, generator-driven representation that preserves both finite
branching structure and asymptotic limit geometry.

By treating analytic generator dynamics as the primary mechanism of
construction, the framework places classical fractal trees within the domain of
smooth analysis without altering their global structure. This makes available
differential and dynamical methods while remaining faithful to established
discrete formulations.

The scope of the present work is deliberately restricted. Generator domains are
one-dimensional, realizations are confined to curve embeddings, and generator
fields are prescribed rather than coupled to external environments or feedback.
These limitations isolate analytic generators as a foundational primitive in
the simplest nontrivial setting and avoid introducing additional structural
assumptions that would obscure the core correspondence results.

\section{Conclusion}

This paper developed a generator-driven framework for the construction of
tree-based fractals in which geometry is realized through the integration of
analytic generator dynamics rather than prescribed through discrete attachment
or symbolic rewriting. Branching is treated as a primitive operation implemented
via exact state inheritance in generator space, ensuring that all finite-depth
realizations consist of smooth curve segments without branchpoint singularities.

Two central results establish a precise correspondence between analytic generator
trees and classical discrete constructions. A combinatorial universality theorem
shows that any discrete tree-based fractal specification can be compiled into an
analytic generator tree whose induced discrete scaffold is isomorphic at every
finite depth. Under standard contractive assumptions, a canopy set equivalence
theorem then shows that the accumulation set of analytic branch endpoints
coincides with the attractor of the corresponding discrete construction.

Together, these results demonstrate that global fractal structure is determined
by recursive branching and scaling independently of the local regularity of
individual branches. Analytic generators therefore provide a smooth compilation
target for classical discrete specifications, preserving both finite
combinatorics and asymptotic limit geometry, and placing tree-based fractals
within a smooth dynamical systems setting.

An important consequence of this correspondence is that any tree-based
fractal admitting a classical discrete specification can be realized as the
limit of a smooth generator-driven dynamical process. While the results of
this paper are purely structural, this realization embeds discrete fractal
trees into a setting where tools from ordinary differential equations and
dynamical systems are, in principle, applicable. In particular, questions
of perturbation, stability under generator modification, and interaction
with external fields may be formulated at the level of generator dynamics
without altering the underlying combinatorial or asymptotic structure.

Although the present work is confined to one-dimensional generator domains and
curve realizations, the framework itself is not inherently restricted to this
setting. The decoupling of internal generator state from spatial projection
suggests natural extensions to reactive generators, in which the field is
modulated by an embedding environment, and to higher-dimensional generator
domains, where state inheritance may support smoothly branching surfaces and
manifolds. These directions are left for future work.

\appendix

\section{Proof of theorem~\ref{thm:combinatorial_universality}}

We first make explicit the minimal data used from a discrete tree-based fractal
specification.

Theorem~\ref{thm:combinatorial_universality} asserts existence of an analytic
generator tree whose induced discrete scaffold is combinatorially identical to
$T_D$ at every finite depth, and (when an embedding is supplied) whose branchpoints
coincide with the embedded node locations.

\subsection{A curve-to-generator lemma (planar case)}

We work in the canonical planar state space
$X=(x,y,\theta)\in\mathbb{R}^3$ with realization map $\Pi(x,y,\theta)=(x,y)$.

\begin{lemma}[Analytic curve induces analytic generator data]
\label{lem:curve_to_generator}
Let $\gamma:[0,S]\to\mathbb{R}^2$ be an analytic regular curve, i.e.\ $\gamma$ is
real-analytic and $\|\gamma'(s)\|>0$ for all $s\in[0,S]$. Define
\[
\rho(s) \;=\; \|\gamma'(s)\|,\qquad
\theta(s) \;=\; \arg(\gamma'(s)),\qquad
\kappa(s) \;=\; \theta'(s).
\]
Then $\rho$ and $\kappa$ are analytic on $[0,S]$, and the solution of
\[
x'=\rho\cos\theta,\quad y'=\rho\sin\theta,\quad \theta'=\kappa,
\qquad (x(0),y(0))=\gamma(0),
\]
satisfies $(x(s),y(s))=\gamma(s)$ for all $s\in[0,S]$.
\end{lemma}

\begin{proof}
Since $\gamma$ is analytic and regular, $\gamma'(s)\neq 0$ on $[0,S]$, so the
argument function $\arg(\gamma'(s))$ can be chosen continuously and is analytic
on the interval (after selecting a consistent branch of the angle). The speed
$\rho(s)=\|\gamma'(s)\|$ is analytic as a composition of analytic functions with
the analytic norm on $\mathbb{R}^2$ away from zero. Then $\kappa=\theta'$ is
analytic. By construction, $\gamma'(s)=\rho(s)(\cos\theta(s),\sin\theta(s))$, so
the planar generator system reproduces $\gamma$ exactly.
\end{proof}

\subsection{Construction of the analytic generator tree}

\begin{construction}[$\mathsf{Compile}(D)$]
\label{cons:compile}
Fix a discrete tree specification $D$ with rooted tree
$T_D=(\mathcal{N}_D,\mathcal{E}_D)$. We build an analytic generator tree
$T_A=(\mathcal{N}_A,\mathcal{E}_A)$ and associated branch realizations as follows.

\begin{enumerate}
\item \textbf{Nodes and edges (combinatorics).}
Set $\mathcal{N}_A := \mathcal{N}_D$ and $\mathcal{E}_A := \mathcal{E}_D$, with
the same root. Thus $T_A$ and $T_D$ are identical as rooted directed graphs.

\item \textbf{Branchpoint locations (embedded case).}
If $D$ supplies an embedding $p:\mathcal{N}_D\to\mathbb{R}^d$, set the realized
branchpoint locations of the analytic construction to match these positions.
In the planar case ($d=2$), we require $\Pi(X_v(0))=p(v)$ for each node $v$.

\item \textbf{Branch segments (geometry along edges).}
For each edge $e=(u\to v)\in\mathcal{E}_A$, choose an analytic regular curve
$\gamma_e:[0,S_e]\to\mathbb{R}^2$ such that
\[
\gamma_e(0)=p(u),\qquad \gamma_e(S_e)=p(v)
\]
(in the embedded case). In the non-embedded case, choose any convenient planar
embedding $p$ consistent with the discrete parameters (e.g.\ a standard recursive
layout); the combinatorial claim does not depend on the particular choice.

\item \textbf{Generator data.}
Use Lemma~\ref{lem:curve_to_generator} to define analytic generator functions
$(\rho_e,\kappa_e)$ (equivalently a generator field $V_e$) on $[0,S_e]$ that
realize $\gamma_e$.

\item \textbf{State inheritance.}
At each node $u$ and for each outgoing edge $e=(u\to v)$, set the initial state
of the child branch to inherit the parent terminal state at the branch event,
so that the realized geometry is continuous at $p(u)$.
\end{enumerate}
\end{construction}

\subsection{Proof of theorem~\ref{thm:combinatorial_universality}}

\begin{proof}[Proof of theorem~\ref{thm:combinatorial_universality}]
We prove the two statements in the theorem.

\paragraph{(i) Combinatorial identity at every finite depth.}
By Construction~\ref{cons:compile}(1), we set $\mathcal{N}_A=\mathcal{N}_D$ and
$\mathcal{E}_A=\mathcal{E}_D$ with the same root and the same directed
parent--child relations. Therefore the identity map
$\varphi:\mathcal{N}_D\to\mathcal{N}_A$ is a rooted-tree isomorphism.
Restricting to any finite depth $k$ yields an isomorphism between the truncated
trees $T_D^{(k)}$ and $T_A^{(k)}$. Hence the induced discrete scaffold of the
analytic construction is combinatorially identical to the discrete tree of $D$
(node-for-node and edge-for-edge in the graph-theoretic sense) at every finite
depth.

\paragraph{(ii) Node-location coincidence in the embedded case.}
Assume $D$ provides an embedding $p:\mathcal{N}_D\to\mathbb{R}^2$ (or that we have
chosen one). For each edge $e=(u\to v)$ we choose an analytic regular curve
$\gamma_e$ connecting $p(u)$ to $p(v)$ (Construction~\ref{cons:compile}(3)).
By Lemma~\ref{lem:curve_to_generator}, there exists analytic generator data
realizing $\gamma_e$ exactly. Consequently, the realized endpoint of the branch
segment for $e$ is $\gamma_e(S_e)=p(v)$. State inheritance at the branchpoint
ensures that all child branches emanate from the same realized point $p(v)$.
By induction over depth, every branchpoint in the analytic construction coincides
with the corresponding discrete node location. This establishes the embedded-node
strengthening.

\paragraph{Compatibility with discrete parameters.}
The discrete parameters (contraction ratios, orientation labels, rule identifiers)
are preserved as labels on edges in the construction and may be used to select
the family of curves $\{\gamma_e\}$ (e.g.\ by choosing lengths or turning
directions consistent with the labels). This does not affect the combinatorial
isomorphism established in (i), and in the embedded case the node locations are
fixed by $p$ regardless of the particular analytic curve chosen between them.
\end{proof}

\section{Proof of canopy set equivalence}

This appendix provides a proof of Theorem~2 stated in Section~5.4.  
We assume familiarity with the analytic generator framework introduced in Sections~2--4 and with the notion of the induced discrete scaffold defined in Section~5.2.

\subsection{Induced discrete scaffold}

Let $\mathcal{T}$ be an analytic fractal tree generated by recursive application of branch events.  
At each finite branching depth $k$, the construction yields a finite set of branchpoints
\[
\mathcal{N}_k = \{ p_\alpha^{(k)} \} \subset \mathbb{R}^d,
\]
where each $p_\alpha^{(k)}$ is the realized image of a generator state inherited at a branch event.

By construction:
\begin{itemize}
\item branchpoints are inherited exactly under the state inheritance rule,
\item parent--child relations are determined solely by the branching schedule in generator space,
\item no geometric offsets are introduced at branching.
\end{itemize}

Thus, the sets $\mathcal{N}_k$ together with their parent--child relations 
define a rooted discrete tree, referred to as the \emph{induced discrete scaffold}.

\begin{lemma}
At each finite depth $k$, the induced discrete scaffold coincides with the 
discrete fractal tree obtained by composing the contraction maps $\{F_i\}$ 
according to the same branching schedule.
\end{lemma}

\begin{proof}
Each branch event applies a generator inheritance rule whose scaling and 
orientation parameters match those of the corresponding contraction map $F_i$. 
Because branchpoints are inherited without offset, the realized location of 
each branchpoint coincides with the image of the root point under the 
corresponding finite composition of maps $F_i$. The combinatorial structure of 
parent--child relations is identical by construction.
\end{proof}

\subsection{Endpoint sets and Hausdorff convergence}

Let $E_k \subset \mathbb{R}^d$ denote the set of endpoints of branches at depth $k$ in the analytic construction, and let $A_k$ denote the corresponding set of points obtained by applying all length-$k$ compositions of the maps $F_i$ to the root point.

By the previous lemma,
\[
E_k = A_k
\]
for all finite depths $k$.

The sequence $\{A_k\}$ converges in the Hausdorff metric to the unique compact attractor $\mathcal{A}$ of the iterated function system $\{F_i\}$.

\subsection{Proof of theorem~2}

\begin{proof}[Proof of theorem~2]
The canopy set $\mathcal{C}$ of the analytic fractal tree is defined as the set of accumulation points of the endpoint sets $E_k$. Since $E_k = A_k$ for all $k$, and since $\{A_k\}$ converges in the Hausdorff sense to $\mathcal{A}$, it follows that
\[
\mathcal{C} = \mathcal{A}.
\]

The smooth interpolation of geometry along analytic generator trajectories does not affect the location of branchpoints or endpoints, and therefore does not alter the accumulation set. The limit geometry is determined entirely by the combinatorial branching structure 
and contraction ratios encoded in the generator inheritance rules.
\end{proof}

\subsection{Interpretation}

This result establishes that analytic generator--defined fractal trees and their discrete counterparts share identical global limit sets. The difference between the two constructions lies exclusively in the local realization of geometry between branchpoints: discrete constructions interpolate edges symbolically, while analytic generators interpolate them smoothly.

Consequently, fractality is shown to be an organizational property of branching and scaling rather than a consequence of local geometric irregularity.

\subsection{Hausdorff convergence of endpoint sets}

Let $E_k \subset \mathbb{R}^d$ denote the set of branch endpoints obtained after $k$ levels of recursive branching in an analytic fractal tree. Assume that each branch event applies a finite family of generator inheritance rules with uniform contraction ratios $\lambda_i \in (0,1)$, and let
\[
\lambda = \max_i \lambda_i < 1.
\]

\begin{lemma}
Under the uniform contraction assumption, the sequence of endpoint sets $\{E_k\}_{k\ge 1}$ converges in the Hausdorff metric to a unique nonempty compact set $\mathcal{C} \subset \mathbb{R}^d$.
\end{lemma}

\begin{proof}[Proof]
By construction, $E_k$ coincides with the set of points obtained by applying 
all length-$k$ compositions of the corresponding contraction maps to the root 
point. Uniform contraction implies that the diameters of these images shrink 
geometrically. Standard results from iterated function system theory therefore 
guarantee that $\{E_k\}$ converges in the Hausdorff sense to a unique compact 
limit set, which is precisely the canopy set $\mathcal{C}$. \qedhere
\end{proof}

This convergence establishes the existence of the canopy set and underpins its 
identification with the attractor of the corresponding discrete fractal 
construction.

\bibliographystyle{amsalpha}
\bibliography{references}

\end{document}